\documentclass[10pt,a4paper]{article}
\usepackage[a4paper]{geometry}
\usepackage{amssymb,latexsym,amsmath,amsfonts,amsthm,currvita}
\usepackage{graphicx,comment}
\usepackage{epsfig}
\usepackage{tikz}
\usepackage{subcaption}
\usepackage{comment}
\usepackage[T1]{fontenc}

\newcommand{\Ls}{\mathcal{L}_s}
\newcommand{\Us}{\mathcal{U}_s}

\newcommand{\No}{\mathbb{N}_{\mathrm{odd}}}

\newtheorem{theorem}{Theorem}[section]
\newtheorem{lemma}[theorem]{Lemma}
\newtheorem{proposition}[theorem]{Proposition}

\theoremstyle{definition}

\theoremstyle{remark}

\numberwithin{equation}{section}

\title{Asymptotics of the optimal values of potentials generated by greedy energy sequences on the unit circle}

\author{Abey L\'{o}pez-Garc\'{i}a \qquad Erwin Mi\~{n}a-D\'{i}az}

\date{\today}

\begin{document}

\maketitle

\begin{abstract}
For the Riesz and logarithmic potentials, we consider greedy energy sequences $(a_n)_{n=0}^\infty$ on the unit circle $S^1$, constructed in such a way that for every $n\geq 1$, the discrete potential generated by  the first $n$ points $a_0,\ldots,a_{n-1}$  of the sequence attains its minimum value (say $U_n$) at $a_n$.  We obtain asymptotic formulae that describe the behavior of $U_n$ as $n\to\infty$, in terms of  certain bounded arithmetic functions with a doubling periodicity property. As previously shown in \cite{LopMc2}, after properly translating and scaling $U_n$, one obtains a new sequence $(F_n)$ that is bounded and divergent. We find the exact value of $\liminf F_n$ (the value of $\limsup F_n$ was already given in \cite{LopMc2}), and show that the interval $[\liminf F_n,\limsup F_n]$ comprises all the limit points of the sequence $(F_n)$.

\smallskip



\end{abstract}

\section{Introduction and  main results}

For $s\geq 0$, let  $K_s:\mathbb{C}\times \mathbb{C}\to (-\infty,\infty]$ be defined as
\[
K_{s}(x,y):=\begin{cases}
|x-y|^{-s}, & \,\,s>0,\\
-\log|x-y|, & \,\, s=0.
\end{cases}
\]
For $s>0$, $K_s$ is known as the Riesz kernel, while for $s=0$, it is known as the logarithmic kernel. 

The $s$-energy of $N\geq  2$ arbitrary points $z_1,\ldots,z_N$ is defined as
\[
E_s(z_1,\ldots,z_N):=\sum_{1\leq j\not=k\leq N}K_s(z_j,z_k)=2\sum_{1\leq j<k\leq N}K_s(z_j,z_k),
\]
and the potential generated by $N\geq 1$ points $z_1,\ldots,z_N$ is the function  
\[
x\mapsto\sum_{\ell=1}^N K_s(x,z_\ell),\qquad x\in \mathbb{C}.
\]
We adopt the convention that the energy $E_s(z_1)$ of a single particle is $0$.

A greedy $s$-energy sequence $(a_{n})_{n=0}^{\infty}\subset S^{1}$ is a sequence of points on the unit circle $S^1$ generated through the following recursive algorithm. The initial point $a_{0}\in S^{1}$ is chosen arbitrarily, and each subsequent point $a_{n}$, $n\geq 1$, is selected so as to minimize the potential generated by the $n$ previously chosen points:
\begin{equation}\label{greedyalg}
\sum_{\ell=0}^{n-1}K_{s}(a_{n},a_{\ell})=\min_{x\in S^1}\sum_{\ell=0}^{n-1}K_{s}(x,a_{\ell}),\qquad n\geq 1.
\end{equation}

If the choice of a point $a_{n}$ is not unique, we pick any among the possible minimizers. In the case of the logarithmic potential, such sequences are more commonly known in the literature as Leja sequences.

Let us denote by $U_{N,s}(x)$ the potential  generated by the first $N$ points $a_{0},\ldots,a_{N-1}$ of a (fixed) greedy $s$-energy sequence $(a_{n})_{n=0}^\infty$. In this way,  
\begin{equation*}
	U_{N,s}(x)=
	\begin{cases}
		\sum_{k=0}^{N-1}|x-a_{k}|^{-s} ,& s>0,\\[0.5em]
		-\sum_{k=0}^{N-1} \log|x-a_{k}|, & s=0.
	\end{cases}
\end{equation*}
Thus, according to \eqref{greedyalg}, $U_{N,s}(x)$ attains its minimum value on $S^{1}$ at the point $a_{N}$. This optimal value is also known in the literature as polarization of the configuration $(a_{0},\ldots,a_{N-1})$, see  e.g. \cite{BorHarSaff}. 

Since 
\[
E_s(a_0,\ldots,a_{N-1},x)=2U_{N,s}(x)+E_s(a_0,\ldots,a_{N-1}),
\]
we see that the  energy  $E_s(a_0,\ldots,a_{N-1},x)$, as a function of  $x\in S^1$, is also minimized when $x=a_N$.

The asymptotic behavior as $N\to\infty$ of $U_{N,s}(a_N)$, as well as that of the energy $E_s(a_0,\ldots,a_N)$, have been investigated in \cite{LopMc2,LopSaff,LopWag}. The results of \cite{LopSaff} are broader in scope, as they are formulated  for more general compact sets, not just $S^1$. In this work, we continue the research initiated in \cite{LopMc2} on the limiting behavior of the sequence $(F_{N,s})_{N=1}^\infty$  defined as 
\begin{align}\label{defF_N,s}\begin{split}
F_{N,s}:=\begin{cases}
	-\frac{U_{N,0}(a_{N})}{\log(N+1)}, & s=0,\\[0.7em]
	\frac{U_{N,s}(a_N)-NI_s(\sigma)}{N^s},& 0<s<1,\\[0.5em]
	\frac{U_{N,1}(a_{N})-\pi^{-1} N\log N}{N}, & s=1,\\[0.5em]
	\frac{U_{N,s}(a_N)}{N^s}, & s>1.
\end{cases}
\end{split}
\end{align}
Here, $I_s(\sigma)$, $0<s<1$, is the energy of the arclength measure $\sigma$ on $S^1$, normalized so that $\sigma(S^{1})=1$ (cf. \cite[Cor. A.11.4]{BorHarSaff}):
\begin{equation}\label{energyeqmeas}
	I_{s}(\sigma):=\iint_{S^{1}\times S^{1}}|x-y|^{-s}\,d\sigma(x)\,d\sigma(y)=\frac{2^{-s}}{\sqrt{\pi}}\frac{\Gamma\left(\frac{1-s}{2}\right)}{\Gamma(1-\frac{s}{2})}.
\end{equation}

 It was shown in \cite{LopMc2} (see Theorems 1.1, 1.4, 1.6, and 1.7 therein) that for each value of $s\geq 0$, the sequence $(F_{N,s})_{N=1}^\infty$ is bounded and divergent. Thus, a study of the asymptotic behavior of the sequence $(F_{N,s})$ is a study of the first-order asymptotics of $U_{N,s}(a_N)$ when $s=0$ or $s>1$, and a study of the second-order asymptotics of $U_{N,s}(a_N)$ when $0<s\leq 1$.
 
 While the exact value of the upper limit of the sequence $(F_{N,s})$ was found in \cite{LopMc2}, the lower limit was expressed in terms of certain extremal quantities whose values were still to be determined. Some of the limit points of the sequence $(F_{N,s})$ were also identified in \cite{LopMc2}. 
 
 We will improve upon these results by finding the exact value of the lower limit $\liminf F_{N,s}$, and showing that the interval $[\liminf F_{N,s},\limsup F_{N,s}]$ comprises all the limit points of the sequence $(F_{N,s})$. Moreover, for $s>0$, we establish a new (finite) asymptotic expansion for $U_{N,s}(a_N)$ that explains more transparently the asymptotic features  exhibited by the sequence $(F_{N,s})$. 

Two quantities that will play an important role in our analysis are the energy 
\[
\mathcal{L}_s(N):=E_s(z_{1,N},\ldots,z_{N,N})  
\]
of the configuration formed by the $N$th roots of unity
\[
z_{k,N}:=\exp(2\pi i(k-1)/N),\qquad 1\leq k\leq N,
\]
 and the minimum value $\Us(N)$ that the potential generated by $z_{1,N},\ldots,z_{N,N}$ attains on $S^1$. This minimum value is attained at any of the points situated midway between two consecutive roots, so that
 \begin{align*}
	\mathcal{U}_{s}(N)=\sum_{k=1}^{N}K_s(e^{-\pi i/N}z_{1,N},z_{k,N}), \qquad s\geq 0.
\end{align*}

It is not difficult to see that the quantities $\Us(N)$ and $\Ls(N)$ are related by the identity (see \cite[Prop. 3.3]{LopMc2})
\begin{equation}\label{eq:relUsLs}
	\Us(N)=\frac{\mathcal{L}_{s}(2N)}{2N}-\frac{\Ls(N)}{N},\qquad N\geq 1.
\end{equation}
 
Due to the circle symmetry, greedy $s$-energy sequences have a very rigid geometric structure, which is the same for all values of $s\geq 0$ (see \cite{BC,LopMc2,LopSaff}). As shown in \cite{BC}, see also \cite[Prop. 3.4]{LopMc2}, every section $(a_{0},\ldots,a_{N-1})$ of a greedy sequence can be described geometrically in terms of the binary representation of $N$. Using that description, it was proved in \cite[Prop. 3.5]{LopMc2} that if $N\in\mathbb{N}$ has the binary decomposition
\begin{equation}\label{bindecompN}
N=2^{n_{1}}+2^{n_{2}}+\cdots+2^{n_{p}},\qquad n_{1}>n_{2}>\cdots>n_{p}\geq 0,
\end{equation}
then\begin{equation}\label{potbinform}
U_{N,s}(a_N)=\sum_{j=1}^{p}\mathcal{U}_{s}(2^{n_{j}}),\qquad s>0.
\end{equation}
 
Formula \eqref{potbinform} is fundamental to our analysis as it allows us to deduce the asymptotic behavior of $U_{N,s}(a_N)$ from that of $\Us(N)$. In turn, the behavior of $\Us(N)$ can be derived  via \eqref{eq:relUsLs} from the asymptotic expansions for $\Ls(N)$ that were established in \cite{BHS}. 

The logarithmic case $s=0$ is much simpler because, as shown by Calvi and Van Manh in \cite{CalviVan}, the representation of $U_{N,0}(a_N)$ in terms of the binary decomposition \eqref{bindecompN} takes the particularly simple form
\begin{equation}\label{CalviVanform}
U_{N,0}(a_N)=-p \log 2.
\end{equation}
Observe that \eqref{potbinform} reduces to \eqref{CalviVanform} when $s=0$, because $\mathcal{U}_0(N)\equiv -\log 2$.

We now introduce the functions that play a primary role in the asymptotic description of $U_{N,s}(a_N)$. For an integer $N\in\mathbb{N}$, we look at its binary decomposition \eqref{bindecompN} and define
\begin{equation}\label{periodfunc}
\begin{aligned}
\eta(N) & :=\left(\frac{2^{n_{1}}}{N},\frac{2^{n_{2}}}{N},\ldots,\frac{2^{n_{p}}}{N}\right),\\
\tau_{b}(N) & :=p.
\end{aligned}
\end{equation}
 Observe that $\tau_{b}(N)$ is the number of 1's in the binary representation of $N$. The most important property of the arithmetic functions defined in \eqref{periodfunc} is the doubling periodicity
\begin{equation}\label{doubperiod}
\begin{aligned}
\eta(N) & =\eta(2N),\\
\tau_{b}(N) & =\tau_{b}(2N).
\end{aligned} 
\end{equation}

Following the notation in \cite{LopMc2}, for a vector $\vec{\theta}=(\theta_{1},\ldots,\theta_{p})$ of  $p\in\mathbb{N}$  positive components $\theta_k> 0$, we define 
\begin{align*}
G(\vec{\theta};s):={} &\sum_{k=1}^{p}\theta_{k}^{s}, \\
\Lambda(\vec{\theta}):={} &\sum_{k=1}^{p}\theta_{k}\log\theta_{k}.
 \end{align*}
Clearly, for all $N\in\mathbb{N}$, we have
\[
G(\eta(N);1)=1,  \qquad G(\eta(N);0)=\tau_{b}(N), 
\]
and by \eqref{doubperiod}, we also have 
\begin{align}
	G(\eta(N);s)={} &G(\eta(2N);s),\qquad  s\geq 0, \label{doublingperiodicity}\\
	 \Lambda(\eta(N))={} & \Lambda(\eta(2N)). \label{doublingperiodicity2}
\end{align}

It was proven in \cite{LopMc2} (Theorems 1.4, 1.6, and 1.7) that for $s>0$, the limit superior and the limit inferior of the sequence \eqref{defF_N,s} are given by  
\begin{align*}
\limsup_{N\rightarrow\infty}F_{N,s}={} &
\begin{cases} (2^{s}-1)\frac{2\zeta(s)}{(2\pi)^{s}},&\quad s>0,\ s\not=1,\\[0.5em]
\frac{1}{\pi}\left(\gamma+\log(8/\pi)\right),&\quad s=1,\\
\end{cases}
\end{align*}
\begin{align}\label{liminfF_N}
	\liminf_{N\rightarrow\infty}F_{N,s} & =\begin{cases}\overline{g}(s)(2^{s}-1)\frac{2\zeta(s)}{(2\pi)^{s}},&\quad 0<s<1,\\[0.5em]
		\frac{1}{\pi}\left(\gamma+\log(8/\pi)+\lambda\right),&\quad s=1,\\[0.5em]
		\underline{g}(s)(2^{s}-1)\frac{2\zeta(s)}{(2\pi)^{s}},&\quad s>1,
	\end{cases}
\end{align}
where $\gamma=\lim_{N\rightarrow\infty}(\sum_{k=1}^{N}\frac{1}{k}-\log N)$ is the Euler-Mascheroni constant, $\zeta(s)$ is the Riemann zeta function, and  
\begin{align}
\overline{g}(s) & :=\sup_{N\in\mathbb{N}} G(\eta(N);s),\qquad 0<s<1,\label{defgup}\\
\underline{g}(s) & :=\inf_{N\in\mathbb{N}} G(\eta(N);s),\qquad s>1,\label{defgdown}\\
\lambda & :=\inf_{N\in\mathbb{N}}\Lambda(\eta(N)).\label{deflambda}
\end{align}  

Moreover, for every $M\in\mathbb{N}$, the value $	G(\eta(M);s)(2^{s}-1) 2\zeta(s)/(2\pi)^s$ is  a limit point of $(F_{N,s})$ when $s\in (0,1)\cup (1,\infty)$, and the value  $\frac{1}{\pi}(\gamma+\log(8/\pi)+\Lambda(\eta(M)))$ is a limit point of $(F_{N,1})$.
  
The  constants $\overline{g}(s)$, $\underline{g}(s)$, and $\lambda$ were estimated in \cite{LopMc2} as follows:
\begin{align*}
\frac{1}{2^s-1}\leq \overline{g}(s)\leq {} &\frac{2^s}{2^s-1},\qquad 0<s<1,\\
0\leq \underline{g}(s)\leq {} &\frac{1}{2^s-1},\qquad s>1,\\
-2e^{-1} -2\log 2\leq  \lambda\leq {} & -2\log 2.
\end{align*}
 Our first result provides the exact values of these constants.
 \begin{theorem}\label{theo:main1}
	The values of the quantities defined in \eqref{defgup}, \eqref{defgdown}, and \eqref{deflambda} are 
	\begin{align}
		\overline{g}(s) & =\frac{1}{2^{s}-1},\qquad 0<s<1,\label{valuegup}\\
		\underline{g}(s) & =\frac{1}{2^{s}-1},\qquad s>1,\label{valuegdown}\\
		\lambda & =-2 \log 2.\label{valuelambda}
	\end{align}
It thus follows from \eqref{liminfF_N} that
\begin{align*}
	\liminf_{N\rightarrow\infty}F_{N,s} & =\begin{cases}\frac{2\zeta(s)}{(2\pi)^{s}},&\quad s>0,\ s\not=1,\\[0.5em]
		\frac{1}{\pi}\left(\gamma+\log(2/\pi)\right),&\quad s=1.
	\end{cases}
\end{align*}
\end{theorem}

In the logarithmic case $s=0$, we have by virtue of \eqref{CalviVanform} that 
\begin{align}\label{explicitformulaforF_N,0}
F_{N,0}=\frac{\tau_{b}(N) \log 2}{\log(N+1)},\qquad N\geq 1.
\end{align}
Using this formula, one can prove without much difficulty that (cf. \cite[Thm. 1.1]{LopMc2})
\begin{align*}
0<F_{N,0}\leq 1,\qquad N\geq 1,
\end{align*}
and that  
\begin{align}\label{liminfsuplogcase}
\liminf_{N\rightarrow\infty}F_{N,0}=0,\qquad \limsup_{N\rightarrow\infty}F_{N,0}=1.
\end{align}

The results discussed so far indicate that when $s>0$, the sequences $G(\eta(N);s)$ and $\Lambda(\eta(N))$, $N\in \mathbb{N}$, exert a strong influence on the limiting behavior of the sequence $(F_{N,s})$. Our next theorem makes clear why this is so. 

\begin{theorem}\label{corollaryonasymptotics}
	The following asymptotic formulae hold true as $N\to\infty$:  If $s>0$ and $s\not=1$, then 
\begin{align}\label{thirdcase}
	\begin{split}
		F_{N,s}	={} & (2^{s}-1)\frac{2\zeta(s)}{(2\pi)^s}G(\eta(N);s)+\begin{cases}
			O(N^{-s}), & 0<s<1,\\
			O(N^{1-s}),& 1<s<3,\ s\not=2,\\
			O(N^{-2}),	& s=2, \ \mathrm{or}\   s>3,\\
			O(N^{-2}\log N), & s=3.
		\end{cases}
	\end{split}
\end{align}If $s=1$, then 
	\begin{align}\label{secondcase}
		\begin{split}
		F_{N,1}= {} &
		\frac{1}{\pi}(\gamma+\log(8/\pi) +\Lambda(\eta(N)))+O(N^{-1}).	
		\end{split}
	\end{align}
\end{theorem}

Theorem \ref{corollaryonasymptotics} shows that when $s>0$, and as $N\to\infty$, $F_{N,s}$ mimics either the behavior  of $G(\eta(N);s)$, or the behavior of  $\Lambda(\eta(N))$, according to  whether $s\not=1$ or not. In particular,  we see from \eqref{thirdcase}, \eqref{secondcase}, and \eqref{doublingperiodicity}--\eqref{doublingperiodicity2}, that the sequence $(F_{N,s})$ has an asymptotic doubling periodicity, by which we mean that
\[
\lim_{N\rightarrow\infty}(F_{2N,s}-F_{N,s})=0.
\]
The latter equality also holds for $s=0$, a fact that can be easily deduced from \eqref{explicitformulaforF_N,0}. 

For every $s\geq 0$, let $T_s$ be the closed interval defined as 
\begin{align}\label{defTs}
T_s:=\left[\liminf_{N\to\infty}F_{N,s},\limsup_{N\to\infty}F_{N,s}\right],
\end{align}
that is,
\[
T_{s}=\begin{cases}
[0,1], &\quad s=0,\\[0.5em]
\left[2(2\pi)^{-s} \zeta(s),2(2^{s}-1)(2\pi)^{-s}\zeta(s)\right], &\quad s>0, s\neq 1,\\[0.5em]
[(1/\pi)(\gamma+\log(2/\pi)), (1/\pi)(\gamma+\log(8/\pi))], &\quad s=1.
\end{cases}
\]
Let us also denote by $F_s'$, $G'_s$, and $\Lambda'$, respectively, the set of limit points of the sequences $(F_{N,s})_{N=1}^\infty$, $(G(\eta(N);s))_{N=1}^\infty$, and $(\Lambda(\eta(N)))_{N=1}^\infty$. By a limit point of a sequence $(x_N)$ we mean any value that is the limit of a subsequence of $(x_N)$.

We will see that for all $N\in\mathbb{N}$, we have the inequalities
\begin{align}
	1 & \leq G(\eta(N);s)< \frac{1}{2^{s}-1},\qquad 0<s<1, \label{ineq1}\\
	\frac{1}{2^{s}-1} & < G(\eta(N);s)\leq 1,\qquad s>1,\label{ineq2}\\[.5em]
	-2 \log 2& <\Lambda(\eta(N))\leq 0.\label{ineq3}
\end{align}
The upper bounds in \eqref{ineq2} and \eqref{ineq3}, as well as the lower bound in \eqref{ineq1}, were obtained already in \cite{LopMc2} and are clearly attained for every $N$ that is a power of $2$. 

By \eqref{doublingperiodicity} and \eqref{doublingperiodicity2}, each value $ G(\eta(N);s)$ and $ \Lambda(\eta(N))$ is repeated infinitely many times as $N\to\infty$, and therefore $ G(\eta(N);s)\in G_s'$ and $\Lambda(\eta(N))\in \Lambda'$ for all $N$. We will employ a density argument to show that $G'_s$ is, in fact, the closed interval with endpoints at $1$ and $(2^s-1)^{-1}$, and that $\Lambda'=[-2\log 2,0]$. 

 Theorem \ref{corollaryonasymptotics} establishes a clear one-to-one correspondence between $F_s'$ and $G_s'$ when $s>0,\ s\not=1$, and between $F_s'$ and $\Lambda'$ when $s=1$. Making use of that correspondence (and of \eqref{explicitformulaforF_N,0} in the case $s=0$) leads to the following theorem.

\begin{theorem} \label{limitpointsthm} For every $s\geq 0$, $F_s'=T_s$.
\end{theorem}

Our last two theorems provide a finite asymptotic expansion for $U_{N,s}(a_N)$ for positive values of $s$. The higher the value of the Riesz parameter $s$, the more terms that occur in the expansion. In particular, by comparing the rate of growth of these terms, one obtains Theorem \ref{corollaryonasymptotics} as an easy corollary.

Let
\[
v_s:=\frac{2^{-s}\,\Gamma(\frac{1-s}{2})}{\sqrt{\pi}\,\Gamma(1-\frac{s}{2})}, \qquad s\in \mathbb{C}\setminus \{1,3,5,\ldots\}.
\]
Note that, in view of \eqref{energyeqmeas},  $v_{s}=I_{s}(\sigma)$ for $s$ in the range $0<s<1$. This function defining $v_s$ is analytic in the indicated domain, and vanishes if $s$ is a positive even integer. 

As is customary, we use $\mathrm{sinc}\,z$ to denote the function
\[
\mathrm{sinc}\,z:=\begin{cases}
	(\sin \pi z)/(\pi z), & \,\,z\neq 0,\\
	1, & \,\,z=0.
\end{cases}
\]
For $s\in\mathbb{C}$ and $|z|<1$, we define $\mathrm{sinc}^{-s}\,z:=\exp(-s \log \mathrm{sinc}\,z)$, where we choose the analytic branch of $\log\mathrm{sinc}\,z$ with value at the origin  $\log\mathrm{sinc}\,0=0$. Let $\alpha_{n}(s)$, $n\geq 0$, denote the coefficients in the Taylor expansion
\[
\mathrm{sinc}^{-s}\,z=\sum_{n=0}^{\infty}\alpha_{n}(s) z^{2n},\qquad |z|<1.
\]
In particular,  $\alpha_{0}(s)=1$.

In what follows, $\mathbb{N}_{\mathrm{odd}}$ denotes the set of all odd positive integers, $\lfloor x\rfloor$ denotes the integer part of $x$, and   $(x)_{n}$ denotes the Pochhammer symbol.

\begin{theorem}\label{expansionthmgeneralcase} For every $s>0$, $s\not \in \No$,
	\begin{equation}\label{expansion1}
		U_{N,s}(a_N)=v_s N+\frac{2}{(2\pi)^s}\sum_{j=0}^{\lfloor s/2\rfloor}\alpha_j(s)\zeta(s-2j)(2^{s-2j}-1)G(\eta(N);s-2j)N^{s-2j}+H_{N,s}\,,
	\end{equation}
	where the sequence $(H_{N,s})_{N=1}^\infty$ is uniformly bounded in $N$. In the special case when $s$ is a positive even integer, $v_s=0$ and $H_{N,s}= 0$ for all $N\geq 1$. 
\end{theorem}

Note that only when $s>2$, the sum $\sum_{j=0}^{\lfloor s/2\rfloor}$ that occurs on the right-hand side of \eqref{expansion1} contains more than one term. This is because when  $s$ is an even positive integer, the term corresponding to $j=\lfloor s/2\rfloor$ reduces to zero.

To formulate the asymptotic expansion corresponding to $s\in \No$, we need an additional piece of notation. For every integer $M\geq 0$, let us define the constant 
\begin{equation}\label{def:CM}
	C_M:=\frac{\alpha'_M(2M+1)}{\alpha_M(2M+1)}+\frac{1}{2}\psi(M+1)-\frac{1}{2}\psi(M+1/2),
\end{equation}
where $\psi=\Gamma'/\Gamma$ is the Digamma function. In this formula, $\alpha_{M}'(2M+1)$ indicates the derivative of $\alpha_{M}(s)$ at $s=2M+1$, and according to Theorem 1.3 in \cite{BHS}, we have
\[
\alpha_{M}'(2M+1)=\sum_{k=0}^{M-1}\alpha_{k}(2M+1)\frac{\zeta(2(M-k))}{M-k}.
\] 
In particular, since $\alpha_0(s)=1$, $\psi(1)=-\gamma$, and $\psi(1/2)=-\gamma-2\log 2$, we have
\[
C_0=\log 2.
\]

\begin{theorem} \label{expansionthmspecialcase} If $s\in \No$, with $s=2M+1$, 
	\begin{align}\label{expansion2}
		\begin{split}
			U_{N,s}(a_N)= {} &
			\frac{(\frac{1}{2})_M}{\pi\,2^{2M}M!}N\log N+	\frac{(\frac{1}{2})_M}{\pi\,2^{2M}M!}(\gamma+\log(4/\pi) +C_M+\Lambda(\eta(N)))N
			\\
			&+\frac{2}{(2\pi)^s}\sum_{j=0}^{M-1}\alpha_j(s)\,\zeta(s-2j)\,(2^{s-2j}-1)\,G(\eta(N);s-2j)\,N^{s-2j}+H_{N,s},
		\end{split}
	\end{align}
	where the sequence $(H_{N,s})_{N=1}^\infty$ is uniformly bounded in $N$. 
\end{theorem}

Observe that only when $s>3$, the sum $\sum_{j=0}^{M-1}$  that occurs in \eqref{expansion2} contains more than one term. 

\section{Proof of the asymptotic formulae}

In what follows, any sum of the form $\sum_{j=J}^{J'}$, with $J'<J$, is to be interpreted as $0$. 

\begin{proof}[\textbf{\emph{Proof of Theorem~\ref{expansionthmgeneralcase}:}}]
We know by Theorem 1.1 in \cite{BHS} that for all $s>0$, $s\not \in \No$, and for every integer $J\geq 0$, we have
\begin{align}\label{asymptoticenergyformula}
	\frac{\Ls(N)}{N}= v_s N+\frac{2}{(2\pi)^s}\sum_{j=0}^J\alpha_j(s)\,\zeta(s-2j)\,N^{s-2j}+O(N^{s-2-2J}),\qquad N\rightarrow\infty,
\end{align}
but when $s$ is a positive even number, $v_s=0$ and \eqref{asymptoticenergyformula} holds in the exact form  
\begin{align}\label{exactenergyformula}
	\frac{\Ls(N)}{N}= \frac{2}{(2\pi)^s}\sum_{j=0}^{s/2}\alpha_j(s)\,\zeta(s-2j)N^{s-2j}.
\end{align}

Combining \eqref{eq:relUsLs} and \eqref{asymptoticenergyformula},  we obtain
\[
\Us(N)=v_sN+\frac{2}{(2\pi)^s}\sum_{j=0}^J\alpha_j(s)\,\zeta(s-2j)\,(2^{s-2j}-1)\,N^{s-2j}+O(N^{s-2-2J}),
\]
which together with \eqref{potbinform} (for $N$ as in \eqref{bindecompN}) yields
\begin{align}\label{expansionUNsaN}
	\begin{split}
U_{N,s}(a_N)={} & v_sN+\frac{2}{(2\pi)^s}\sum_{j=0}^J\alpha_j(s)\,\zeta(s-2j)\,(2^{s-2j}-1)\sum_{k=1}^p2^{n_k(s-2j)}\\
&+O\left(\sum_{k=1}^p2^{n_k(s-2-2J)}\right).
\end{split}
\end{align}

Since $s>0$, we have $\lfloor s/2\rfloor \geq 0$, and taking  $J=\lfloor s/2\rfloor$ in \eqref{expansionUNsaN} we obtain
\begin{align*}
	\begin{split}
	U_{N,s}(a_N)={} & v_sN+\frac{2}{(2\pi)^s}\sum_{j=0}^{\lfloor s/2\rfloor}\alpha_j(s)\zeta(s-2j)(2^{s-2j}-1)G(\eta(N);s-2j)N^{s-2j}+H_{N,s},
\end{split}
\end{align*}
with 
\[
H_{N,s}=O\left(\sum_{k=1}^p2^{n_k(s-2-2J)}\right).
\]
Since $s-2-2J<0$, it follows that $H_{N,s}=O(1)$ as $N\to\infty$, which proves \eqref{expansion1}.

If we repeat the above chain of arguments with  \eqref{exactenergyformula} in place of  \eqref{asymptoticenergyformula}, we  conclude that for $s$ a positive even integer, \eqref{expansion1} holds with $v_s=0$ and $H_{N,s}\equiv 0$.
\end{proof}

\begin{proof}[\textbf{\emph{Proof of Theorem \ref{expansionthmspecialcase}:}}]
	If $s=2M+1$ is an odd positive integer, then by Theorem 1.3 in \cite{BHS}, we know that for any $J>M$ the following expansion holds as $N\to\infty$:
	\begin{align}\label{firstexpansionSodd}\begin{split}
		\frac{\Ls(N)}{N}={} &\frac{(\frac{1}{2})_M}{\pi\, 2^{2M}M!}N\log N +
		\frac{(\frac{1}{2})_M}{\pi\,2^{2M}M!}(\gamma-\log \pi +C_M)N\\
		&+\frac{2}{(2\pi)^s}\underset{j\not=M}{\sum_{j=0}^{J}}\alpha_j(s)\,\zeta(s-2j)\,N^{s-2j}+O(N^{s-2-2J}).
		\end{split}
	\end{align}
Here $C_{M}$ is the constant defined in \eqref{def:CM} and $\gamma$ is the Euler-Mascheroni constant. Combining \eqref{firstexpansionSodd} and \eqref{eq:relUsLs} we get
	\begin{align*}
		\Us(N)= {} &	\frac{(\frac{1}{2})_M}{\pi\, 2^{2M}M!}N\log N+	\frac{(\frac{1}{2})_M}{\pi\,2^{2M}M!}(\gamma+\log(4/\pi) +C_M)N\\
		&+\frac{2}{(2\pi)^s}\underset{j\not=M}{\sum_{j=0}^{J}}\alpha_j(s)\,\zeta(s-2j)\,(2^{s-2j}-1)\,N^{s-2j}+O(N^{s-2-2J}).
	\end{align*}
This expansion for $\Us(N)$ can then be used in \eqref{potbinform} to obtain
		\begin{align}\label{specialcaseexpansion1}
			\begin{split}
		U_{N,s}(a_N)= {} &
		\frac{(\frac{1}{2})_M}{\pi\, 2^{2M}M!}\sum_{k=1}^p2^{n_k}\log (2^{n_k})+	\frac{(\frac{1}{2})_M}{\pi\,2^{2M}M!}(\gamma+\log(4/\pi) +C_M)N+
		\\
		&+\frac{2}{(2\pi)^s}\underset{j\not=M}{\sum_{j=0}^{J}}\alpha_j(s)\,\zeta(s-2j)\,(2^{s-2j}-1)\,\sum_{k=1}^p2^{n_k(s-2j)}+O\left(\sum_{k=1}^p2^{n_k(s-2-2J)}\right).
		\end{split}
	\end{align}

Let us take  $J=M+1$ in \eqref{specialcaseexpansion1}. Since 
	\[
	\sum_{k=1}^p2^{n_k}\log (2^{n_k})=N\sum_{k=1}^p\frac{2^{n_k}}{N}\log \left(\frac{2^{n_k}}{N}\right)+N\log N=N\Lambda(\eta(N))+N\log N\,,
	\]
	we can write \eqref{specialcaseexpansion1} in the form 
	\begin{align*}
		U_{N,s}(a_N)= {} &
		\frac{(\frac{1}{2})_M}{\pi\, 2^{2M}M!}N\log N+	\frac{(\frac{1}{2})_M}{\pi\,2^{2M}M!}(\gamma+\log(4/\pi) +C_M+\Lambda(\eta(N)))N+
		\\
		&+\frac{2}{(2\pi)^s}\sum_{j=0}^{M-1}\alpha_j(s)\,\zeta(s-2j)\,(2^{s-2j}-1)\,G(\eta(N);s-2j)\,N^{s-2j}+H_{N,s},	
	\end{align*}
with
\[
H_{N,s}=O\left(\sum_{k=1}^p2^{n_k(s-2-2M)}\right).
\] 
Since $s-2-2M=-1$, we have $H_{N,s}=O(1)$, and Theorem \ref{expansionthmspecialcase} follows.
\end{proof}

\begin{proof}[\textbf{\emph{Proof of Theorem \ref{corollaryonasymptotics}:}}]  For $0<s<1$, the sum $\sum_{j=0}^{\lfloor s/2\rfloor}$ in \eqref{expansion1} consists of just one term, and since $H_{N,s}=O(1)$ as $N\to\infty$, it follows that \eqref{thirdcase} is just a different way of writing \eqref{expansion1}.
	
	Similarly, \eqref{secondcase} is what \eqref{expansion2} becomes for $s=1$, since for such a value of $s$, $M=0$, so that the sum $\sum_{j=0}^{M-1}$ in \eqref{expansion2} reduces to zero and the constant $C_0=\log 2$.
 
For $1<s<3$, we get from \eqref{expansion1} 
\begin{align*}
	\begin{split}
		\frac{U_{N,s}(a_N)}{N^s}={} & (2^{s}-1)\,\frac{2\zeta(s)}{(2\pi)^s}\,G(\eta(N);s)+v_s\,N^{1-s}+\begin{cases}O(N^{-s}),& 1<s<2,\\
			O(N^{-2}),	& 2\leq s <3.
		\end{cases}
	\end{split}
\end{align*}
Having in mind that $v_s=0$ when $s=2$, the latter equation yields \eqref{thirdcase} for $1<s<3$. The validity of \eqref{thirdcase} for the remaining values of $s$ follows similarly by comparing the order of the terms that occur in \eqref{expansion1} and \eqref{expansion2}. 
\end{proof}

\section{Inequalities and proof of Theorem~\ref{theo:main1}}

To prove Theorem~\ref{theo:main1}, we need to establish the veracity of \eqref{valuegup}--\eqref{valuelambda}. These identities are included in the theses of two stronger results, Theorem \ref{theo:lambda} and Theorem \ref{rangeofG}, to be stated and proved in what follows.

 Recall that for an integer $N\in\mathbb{N}$, we have defined
\[
\eta(N):=\left(\frac{2^{n_{1}}}{N},\frac{2^{n_{2}}}{N},\ldots,\frac{2^{n_{p}}}{N}\right),\qquad \tau_b(N):=p,
\]
if $N$ has the binary decomposition \eqref{bindecompN}. For instance, since $2^p-1=\sum_{k=0}^{p-1}2^k$, we have 
\[
\tau_b(2^p-1)=p
\]
and 
\begin{align}\label{vector-for-eta(2^p-1)}
	\eta(2^p-1)=\left(\frac{2^{p-1}}{2^p-1},\frac{2^{p-2}}{2^p-1},\ldots,\frac{1}{2^p-1}\right).
\end{align}

We note that for all $p\geq 1$,
\[
\min\{N: \tau_b(N)=p\}=2^p-1,
\] 
and that 
\begin{align}\label{extremalvaluep}
	G(\eta(2^p-1);s)=\frac{2^{sp}-1}{(2^s-1)(2^p-1)^s}.
\end{align}

\begin{theorem}\label{theo:lambda}
For any $N\in\mathbb{N}$, we have
\begin{equation}\label{ineq:log}
\Lambda(\eta(N))> -2 \log 2,
\end{equation}
and the identity \eqref{valuelambda} holds.  
\end{theorem}
\begin{proof}
Let $N\in\mathbb{N}$, and suppose that \eqref{bindecompN} is the binary decomposition of $N$. Let $\vec{\theta}=(\theta_{1},\ldots,\theta_{p})=\eta(N)$. We need to show that $\Lambda(\vec{\theta})>\log(1/4)$, or equivalently, that
\[
\prod_{k=1}^{p}\theta_{k}^{\theta_{k}}> \frac{1}{4}.
\]
Since $\sum_{k=1}^{p}\theta_{k}=1$, we have
\[
\prod_{k=1}^{p}\theta_{k}^{\theta_{k}}=\prod_{k=1}^{p}\left(\frac{2^{n_{k}}}{N}\right)^{\theta_{k}}=\frac{2^{\sum_{k=1}^{p} \frac{n_{k} 2^{n_{k}}}{N}}}{N},
\]
so we need to justify the inequality
\begin{equation}\label{ineqtoprove}
2^{\sum_{k=1}^{p} \frac{n_{k} 2^{n_{k}}}{N}}> \frac{N}{4}.
\end{equation}
From the convexity of the function $x\mapsto x\log x$ it is easy to deduce that for any positive numbers $(b_{k})_{k=1}^{p}$ and $(r_{k})_{k=1}^{p}$ we have
\[
\frac{\sum_{k=1}^{p} r_{k} b_{k}}{\sum_{k=1}^{p} r_{k}}\leq \exp\left(\frac{\sum_{k=1}^{p}r_{k} b_{k} \log(b_{k})}{\sum_{k=1}^{p}r_{k} b_{k}}\right).
\]
In this inequality we take $b_{k}=2^{n_{k}+k-2}$ and $r_{k}=2^{-k+2}$, and obtain
\begin{align*}
\frac{N}{4}=\frac{\sum_{k=1}^{p} 2^{n_{k}}}{4}  & < \frac{\sum_{k=1}^{p} 2^{n_{k}}}{4 \sum_{k=1}^{p} 2^{-k}}\leq \exp\left(\frac{\sum_{k=1}^{p} 2^{n_{k}}\log(2^{n_{k}+k-2})}{\sum_{k=1}^{p} 2^{n_{k}}}\right)\\
& =\exp\left(\log(2)\frac{\sum_{k=1}^{p}(n_{k}+k-2) 2^{n_{k}}}{\sum_{k=1}^{p}2^{n_{k}}}\right)\\
& =2^{\frac{\sum_{k=1}^{p}(n_{k}+k-2) 2^{n_{k}}}{N}},
\end{align*}
so if we prove $\sum_{k=1}^{p}(k-2) 2^{n_{k}}\leq 0$, or equivalently
\[
\sum_{k=3}^{p}(k-2) 2^{n_{k}}\leq 2^{n_{1}},
\]
then \eqref{ineqtoprove} is justified. Using the inequality $n_{k}+1\leq n_{k-1}$, we get
\[
\sum_{k=3}^{p}(k-2) 2^{n_{k}}=\sum_{k=3}^{p}\sum_{j=k}^{p} 2^{n_{j}}\leq \sum_{k=3}^{p}\sum_{i=0}^{n_{k}} 2^{i}\leq \sum_{k=3}^{p} 2^{n_{k}+1}\leq \sum_{k=3}^{p} 2^{n_{k-1}}\leq 2^{n_{1}}.
\]
This finishes the proof of \eqref{ineq:log}.

Now we justify \eqref{valuelambda}. From \eqref{deflambda} and \eqref{ineq:log} we obtain $\lambda\geq -2\log 2$. The reverse inequality was justified in \cite[Remark 5.4]{LopMc2}, but we give here a simpler argument. Consider the sequence $N_{p}=2^{p}-1$, $p\geq 1$. Simple arithmetic shows 
\[
\Lambda(\eta(N_{p}))=-\frac{2^{p}}{2^{p}-1}\log (2)\sum_{k=1}^{p}\frac{k}{2^{k}}-\log(1-2^{-p}),
\]
which approaches $-2\log 2$ as $p\rightarrow\infty$. Therefore $\lambda=-2\log 2$.
\end{proof}

For the proof of our next result, we will use the well-known Karamata's inequality for convex functions, see \cite{HLP,KDLM,Karamata}. For convenience of the reader, we cite this classical inequality. Suppose that $(a_{i})_{i=1}^{p}$ and $(b_{i})_{i=1}^{p}$ are two collections of real numbers on an interval $I\subset\mathbb{R}$, satisfying the following properties:
\begin{align*}
a_{1}\geq a_{2}\geq \cdots & \geq a_{p},\qquad b_{1}\geq b_{2}\geq \cdots\geq b_{p},\\
\sum_{i=1}^{k}a_{i}& \geq \sum_{i=1}^{k}b_{i},\qquad 1\leq k\leq p-1,\\
\sum_{i=1}^{p}a_{i} & =\sum_{i=1}^{p}b_{i}.
\end{align*}
If $f:I\rightarrow\mathbb{R}$ is convex, then
\begin{equation}\label{ineqKar}
\sum_{i=1}^{p}f(a_{i})\geq \sum_{i=1}^{p}f(b_{i}).
\end{equation}
Moreover, if $f$ is strictly convex on the interval $I$, then equality in \eqref{ineqKar} holds if and only if $a_{i}=b_{i}$ for all $1\leq i\leq p$. In the case that $f$ is concave or strictly concave, the inequality in \eqref{ineqKar} is reversed. 

\begin{theorem}\label{rangeofG} 
Let $p\geq 2$ be an integer. If $N\in\mathbb{N}$ is such that $2^p-1<N<2(2^{p}-1)$ and $\tau_b(N)=p$, then 
	\begin{align}
G(\eta(N);s)<{} &G(\eta(2^p-1);s),\qquad 0<s<1,\label{inequalityGs<1}\\
G(\eta(N);s)>{} & G(\eta(2^p-1);s),\qquad s>1.\label{inequalityGs>1}
\end{align}
As a consequence, for every $N\in\mathbb{N}$ we have
\begin{align}
1 & \leq G(\eta(N);s)< \frac{1}{2^{s}-1},\qquad 0<s<1,\label{Rieszineq1bis}\\
\frac{1}{2^{s}-1} & < G(\eta(N);s)\leq 1,\qquad s>1.\label{Rieszineq2bis}
\end{align}
The identities \eqref{valuegup} and \eqref{valuegdown} are valid.
\end{theorem}
\begin{proof}  
Let  $N=	\sum_{k=1}^p2^{n_k}$  be as in \eqref{bindecompN}, and assume that $2^p-1<N<2(2^{p}-1)$. We define
	\[
	a_k:=\frac{2^{n_k}}{N},\qquad b_k:=\frac{2^{p-k}}{2^p-1},\qquad 1\leq k\leq p.
	\]
	Then, \eqref{inequalityGs<1} and \eqref{inequalityGs>1} say that 
	\begin{align*}
		\begin{split}
		\sum_{k=1}^p (a_k)^s<{} &	\sum_{k=1}^p (b_k)^s,\qquad 0<s<1,\\
			\sum_{k=1}^p (a_k)^s>{} &	\sum_{k=1}^p (b_k)^s,\qquad s>1.
		\end{split}
	\end{align*} 
	
First we justify the inequality $a_{p}<b_{p}$, which is  equivalent to $2^{p}-1<2^{-n_{p}}N$. We have 
\[
2^{-n_{p}}N=\sum_{k=1}^{p}2^{n_{k}-n_{p}}\geq \sum_{k=0}^{p-1}2^{k}=2^{p}-1.
\]
But it is impossible that $2^{-n_{p}}N=2^{p}-1$ since we are assuming that $2^{p}-1<N<2(2^{p}-1)$.

Clearly we have $\sum_{k=1}^pa_k=\sum_{k=1}^p b_k=1$. Therefore, if we show that  
	\begin{align}\label{majorization}
		\sum_{i=1}^ka_i\geq \sum_{i=1}^kb_i, \qquad 1\leq k\leq p-1,
	\end{align}
	then \eqref{inequalityGs<1} and \eqref{inequalityGs>1} will follow from Karamata's inequality, due to the function $x\mapsto x^s$ being strictly concave for $0<s<1$ and strictly convex for $s>1$.  Since $a_{p}\neq b_{p}$, we will have strict inequality.  
	
	Using that  $\sum_{i=1}^kb_i=2^{p-k}(2^k-1)/(2^p-1)$, and that $N=\sum_{i=1}^k2^{n_i} +	\sum_{i=k+1}^p2^{n_i}$, we can write the inequality   \eqref{majorization} in the equivalent form
	\begin{align}\label{majorization2}
		\sum_{i=1}^k2^{n_i} \geq \frac{2^{p-k}(2^k-1)}{2^{p-k}-1}\sum_{i=k+1}^p2^{n_i}, \qquad 1\leq k\leq p-1.
	\end{align}
	
	From the monotonicity of the $n_k$'s, we see that
	\begin{align*}
		\sum_{i=1}^k2^{n_i} \geq 	\sum_{i=1}^k2^{n_k+i-1}=2^{n_k}(2^k-1),
	\end{align*}
	and that
	\begin{align*}
		\sum_{i=k+1}^p2^{n_i}\leq \sum_{i=k+1}^p2^{n_k+k-i}	=\frac{2^{n_k}(2^{p-k}-1)}{2^{p-k}}.
	\end{align*}
	The combination of these two inequalities readily yields \eqref{majorization2}.
	
The lower bound in \eqref{Rieszineq1bis} and the upper bound in \eqref{Rieszineq2bis} were justified in equations (4.5) and (6.4) in \cite{LopMc2}. As a function of $p$, the expression \eqref{extremalvaluep} is strictly increasing for $0<s<1$, and strictly decreasing for $s>1$, and converges to $(2^{s}-1)^{-1}$. So the two remaining inequalities in \eqref{Rieszineq1bis} and \eqref{Rieszineq2bis} are obtained by letting $p\to \infty$ in \eqref{inequalityGs<1} and \eqref{inequalityGs>1}. The identities \eqref{valuegup} and \eqref{valuegdown} follow immediately.
\end{proof}

\section{Proof of Theorem \ref{limitpointsthm}}
The arguments that we employ to prove Theorem \ref{limitpointsthm} vary in some technical aspects  depending on  whether  $s=0$, $s=1$, or $s\in (0,1)\cup(1,\infty)$. Accordingly, we will split the proof in three different subsections. The following observation, however, is used across the board. 

The set $X'$ of limit points of a bounded sequence $(x_{N})_{N=1}^\infty$ of real numbers is always a closed set that is contained in the interval $T=[\liminf x_N,\limsup x_N]$. Therefore, in order to prove that $X'=T$, it suffices to show that $X'$ contains a subset $X$ whose closure $\overline{X}=T$.

\subsection{Case $s=0$} 
By \eqref{liminfsuplogcase} and the definition of $T_s$ in \eqref{defTs}, we have $T_0=[0,1]$.
The closure of the set $F_0:=\mathbb{Q}\cap (0,1)$ is the interval $T_0$, and so it suffices to show that every element of $F_0$  is a limit point of the sequence $(F_{N,0})$.

Let us fix integers $r$ and $q$ such that $0<r<q$, and consider values of $N$  of the form 
\[
N=2^{nq}+2^{nr-1}+2^{nr-2}+\cdots+2+1,\qquad n\in \mathbb{N},\quad nr\geq 1.
\]  
Then, $\tau_b(N)=nr+1$ and $N+1=2^{nq}+2^{nr}$, so that for $N$
of said form (recall \eqref{explicitformulaforF_N,0}) 
\[
F_{N,0}=\frac{nr+1}{nq+\log_2(1+2^{-n(q-r)})}\xrightarrow[n\to\infty]{}\frac{r}{q}.
\]

\subsection{Case $s>0$, $s\not=1$}\label{sec:density1}

Let the sets of limit points of the sequences $\left(F_{N,s}\right)$ and $\left(G(\eta(N);s)\right)_{N=1}^\infty$ be respectively denoted by  $F_s'$ and $G_s'$. Let us also denote by $\tau_s$ the interval with endpoints at the limit superior and the limit inferior of the sequence $\left(G(\eta(N);s)\right)_{N=1}^\infty$. Then,  
\begin{align*}
	\tau_s= \begin{cases} [1,(2^s-1)^{-1}],& 0<s<1,\\
		[(2^s-1)^{-1},1],& s>1.
	\end{cases}
\end{align*} 

We need to show that $F_s'=T_s$. From the asymptotic formula \eqref{thirdcase}, we know that 
\[ 
\lim_{N\to\infty}\left(F_{N,s}-(2^{s}-1)\frac{2\zeta(s)}{(2\pi)^s}G(\eta(N);s)\right)=0.
\]
This implies that  $F_s'$ can be  obtained by scaling $G_s'$ by a factor of $(2^{s}-1)\frac{2\zeta(s)}{(2\pi)^s}$, so that $F_s'=T_s\iff G_s'=\tau_s$. 

Because of the doubling periodicity property \eqref{doublingperiodicity}, the set of values 
	\[
	G_s:=\{G(\eta(N);s):N\in \mathbb{N}\}
	\]
	is a subset of $G_s'$, and we have $G_{s}'\subset\tau_{s}$. So in order to prove that $G_s'=\tau_s$, it is sufficient  to show that 	
\begin{align}\label{densityfors<1}\overline{G_s}=\tau_s.
\end{align}

In order to prove \eqref{densityfors<1}, it is convenient to work with a different representation for $G(\eta(N);s)$. Using the binary decomposition \eqref{bindecompN}, we can write 
\begin{align}\label{lambdarepresentation}
	G(\eta(N);s)=\frac{\sum_{k=1}^p2^{sn_k}}{\left(\sum_{k=1}^p2^{n_k}\right)^s}=\frac{\sum_{k=1}^p2^{-s\lambda_k}}{\left(\sum_{k=1}^p2^{-\lambda_k}\right)^s},
\end{align}
where $\lambda_k={n_1-n_k}$, so that $0= \lambda_1<\cdots <\lambda_p$.  In particular, we see from \eqref{vector-for-eta(2^p-1)} that for every integer $K\geq 1$,  
\begin{align}\label{fullsum}
	G(\eta(2^K-1);s)=	\frac{\sum_{k=0}^{K-1}2^{-sk}}{\left(\sum_{k=0}^{K-1}2^{-k}\right)^s}.
\end{align}

Conversely, given integers   $0= \lambda_1<\cdots <\lambda_p$, the number $N=\sum_{k=1}^p2^{n_k}$, with $n_k=\lambda_p-\lambda_k$, satisfies \eqref{lambdarepresentation}. Thus,
\[
G_s=\left\{\frac{\sum_{k=1}^p2^{-s\lambda_k}}{\left(\sum_{k=1}^p2^{-\lambda_k}\right)^s}: 0= \lambda_1<\cdots <\lambda_p,\,\,\lambda_{i}\in\mathbb{Z}\right\},
\]
and \eqref{densityfors<1} is clearly implied by the following result.
\begin{proposition}\label{propdensityfors<1} Let $t$ be an interior point of $\tau_s$ such that $t\not\in G_s$. For each $n\in \mathbb{N}$, there exist integers $p_n\geq 1 $, and $0= \lambda_{n,1}<\cdots <\lambda_{n,p_n}$, with $p_n<p_{n+1}$ for all $n\geq 1$, such that  when $0<s<1$, 
	\begin{align}\label{inequalitydenseset-s<1}
			\frac{\sum_{k=1}^{p_n}2^{-s\lambda_{n,k}}}{\left(\sum_{k=1}^{p_n}2^{-\lambda_{n,k}}\right)^s}<t<	\frac{\sum_{k=1}^{p_n}2^{-s\lambda_{n,k}}+2^{-s(\lambda_{n,p_n}+1)}}{\left(\sum_{k=1}^{p_n}2^{-\lambda_{n,k}}+2^{-(\lambda_{n,p_n}+1)}\right)^s},
	\end{align}
and when $s>1$, \eqref{inequalitydenseset-s<1} holds true with the inequalities reversed.  As a consequence, 
	\begin{align}\label{inequalitydenseset2}
	\lim_{n\to\infty}\frac{\sum_{k=1}^{p_n}2^{-s\lambda_{n,k}}}{\left(\sum_{k=1}^{p_n}2^{-\lambda_{n,k}}\right)^s}=t.
\end{align}
\end{proposition}

For the proof of Proposition \ref{propdensityfors<1}, we first  establish a lemma.	
\begin{lemma}\label{inequalitylemma} For any finite collection of  $p\geq 1$ integers  $0\leq  \lambda_1<\cdots <\lambda_p$, the function  
	\begin{align*}
f(x):=\frac{\sum_{k=1}^p2^{-s\lambda_k}+x^s}{\left(\sum_{k=1}^p2^{-\lambda_k}+x\right)^s}
	\end{align*}
is,  on the interval $[0,2^{-(\lambda_p+1)}]$, strictly increasing when $0<s<1$, and strictly decreasing when $s>1$. Also,
\begin{align}\label{differencevaluesf(x)}
	|f(2^{-(\lambda_p+1)})-f(0)|\leq \begin{cases}2^{-s(\lambda_p+1)},& 0<s<1,\\
		s2^{-(\lambda_p+1)}, & s>1.
		\end{cases}
	\end{align}
\end{lemma}
\begin{proof}Set  $A:=\sum_{k=1}^p2^{-s\lambda_k}$ and $B:=\sum_{k=1}^p2^{-\lambda_k}$, so that 
	\[
	f(x)=\frac{A+x^s}{(B+x)^s}, \quad f'(x)=\frac{s(Bx^{s-1}-A)}{(B+x)^{s+1}},\quad x>0.
	\]
	
Let us consider first the case $0<s<1$. Because of the monotonicity of the sequence $(\lambda_k)_{k=1}^{p}$, and the fact that $s-1<0$, we have 
	\begin{align*}
	2^{-(s-1)(\lambda_{p}+1)}2^{-\lambda_k}>2^{-s\lambda_k}, \quad 1\leq k\leq p
	\end{align*}
(this being equivalent to $ 2^{-(s-1)(\lambda_{p}+1)}>2^{-(s-1)\lambda_{k}}$). It follows that 
\[
2^{-(s-1)(\lambda_{p}+1)}\sum_{k=1}^p2^{-\lambda_k}>\sum_{k=1}^p2^{-s\lambda_k},
\] 
or what is the same, that $2^{-(s-1)(\lambda_{p}+1)}B>A$. This implies that $f'(x)>0$ for $0<x<2^{-(\lambda_p+1)}$, and so $f$ increases on $[0,2^{-(\lambda_p+1)}]$.

In an analogous way, one proves that for $s>1$,  $2^{-(s-1)(\lambda_{p}+1)}B<A$, so that $f$ decreases in $[0,2^{-(\lambda_p+1)}]$.  

We now prove \eqref{differencevaluesf(x)}. If $0<s<1$, then 
\begin{align*}
	|f(2^{-(\lambda_p+1)})-f(0)|={} &\frac{\sum_{k=1}^p2^{-s\lambda_k}+2^{-s(\lambda_p+1)}}{\left(\sum_{k=1}^p2^{-\lambda_k}+2^{-(\lambda_p+1)}\right)^s}- \frac{\sum_{k=1}^p2^{-s\lambda_k}}{\left(\sum_{k=1}^p2^{-\lambda_k}\right)^s}\\
	\leq {} &\frac{\sum_{k=1}^p2^{-s\lambda_k}+2^{-s(\lambda_p+1)}}{\left(\sum_{k=1}^p2^{-\lambda_k}\right)^s}- \frac{\sum_{k=1}^p2^{-s\lambda_k}}{\left(\sum_{k=1}^p2^{-\lambda_k}\right)^s}\leq 2^{-s(\lambda_p+1)}.
\end{align*}

For $s>1$, we argue differently. By the Mean Value Theorem, there is  $x_0\in (0,2^{-(\lambda_p+1)})$ such that 
\begin{align*}
	|f(2^{-(\lambda_p+1)})-f(0)|={} &	2^{-(\lambda_p+1)}|f'(x_0)|.
	\end{align*}
From this equality we then derive \eqref{differencevaluesf(x)} by noticing that for every $x\in (0,2^{-(\lambda_p+1)})$,
\[
|f'(x)|= \frac{s(A-Bx^{s-1})}{(B+x)^{s+1}}\leq \frac{sA}{B}\leq s\,.
\]
\end{proof}		

\begin{proof}[\textbf{\emph{Proof of Proposition \ref{propdensityfors<1}:}}] The proof is by induction on $n$, and we only give it for $0<s<1$, since the case $s>1$ is proven similarly.  The sequence $(G(\eta(2^K-1);s))_{K=1}^\infty$ has for first element $1$, and by \eqref{fullsum} and Lemma \ref{inequalitylemma}, this sequence increases and converges to $(2^s-1)^{-1}$.  Since $1< t<(2^s-1)^{-1}$ and $t\not\in G_s$,  there must exist  an integer $K_1\geq 0$ such that (see \eqref{fullsum})
\begin{align*}
		\frac{\sum_{k=0}^{K_1}2^{-sk}}{\left(\sum_{k=0}^{K_1}2^{-k}\right)^s}< t<	\frac{\sum_{k=0}^{K_1+1}2^{-sk}}{\left(\sum_{k=0}^{K_1+1}2^{-k}\right)^s}\,.
\end{align*}
This implies that the inequality \eqref{inequalitydenseset-s<1} is satisfied for $n=1$  with the choice of $p_1=K_1+1$ and $\lambda_{1,k}=k-1$ for $1\leq k\leq p_1$. 

Let us assume now that \eqref{inequalitydenseset-s<1} holds for some integer $n\geq 1$. By Lemma \ref{inequalitylemma}, the values 
\[
\frac{\sum_{k=1}^{p_n}2^{-s\lambda_{n,k}}+2^{-s(\lambda_{n,p_n}+j)}}{\left(\sum_{k=1}^{p_n}2^{-\lambda_{n,k}}+2^{-(\lambda_{n,p_n}+j)}\right)^s}, \quad j=1, 2,\ldots
\]
strictly decrease as $j\to\infty$ towards the fraction on the left of \eqref{inequalitydenseset-s<1}, and so there must exist an integer  $\lambda\geq \lambda_{n,p_n}+2$ such that  
\begin{align*}
	\frac{\sum_{k=1}^{p_n}2^{-s\lambda_{n,k}}+2^{-s\lambda}}{\left(\sum_{k=1}^{p_n}2^{-\lambda_{n,k}}+2^{-\lambda}\right)^s}< t<	\frac{\sum_{k=1}^{p_n}2^{-s\lambda_{n,k}}+2^{-s(\lambda-1)}}{\left(\sum_{k=1}^{p_n}2^{-\lambda_{n,k}}+2^{-(\lambda-1)}\right)^s}.
\end{align*}

Since $(2^s-1)^{-1}>1$, we have
\begin{align*}
	\begin{split} \frac{\sum_{k=1}^{p_n}2^{-s\lambda_{n,k}}+\sum_{j=0}^\infty2^{-s(\lambda+j)}}{\left(\sum_{k=1}^{p_n}2^{-\lambda_{n,k}}+\sum_{j=0}^\infty 2^{-(\lambda+j)}\right)^s}={} &	\frac{\sum_{k=1}^{p_n}2^{-s\lambda_{n,k}}+2^{-s(\lambda-1)}(2^s-1)^{-1}}{\left(\sum_{k=1}^{p_n}2^{-\lambda_{n,k}}+2^{-(\lambda-1)}\right)^s}\\
		>{} &\frac{\sum_{k=1}^{p_n}2^{-s\lambda_{n,k}}+2^{-s(\lambda-1)}}{\left(\sum_{k=1}^{p_n}2^{-\lambda_{n,k}}+2^{-(\lambda-1)}\right)^s},
	\end{split}
\end{align*}
so that by Lemma \ref{inequalitylemma} there must exist an integer $m\geq 0$ such that 
\begin{align}\label{casen+1}
	\frac{\sum_{k=1}^{p_n}2^{-s\lambda_{n,k}}+\sum_{j=0}^m2^{-s(\lambda+j)}}{\left(\sum_{k=1}^{p_n}2^{-\lambda_{n,k}}+\sum_{j=0}^m2^{-(\lambda+j)}\right)^s}< t<	\frac{\sum_{k=1}^{p_n}2^{-s\lambda_{n,k}}+\sum_{j=0}^{m+1}2^{-s(\lambda+j)}}{\left(\sum_{k=1}^{p_n}2^{-\lambda_{n,k}}+\sum_{j=0}^{m+1}2^{-(\lambda+j)}\right)^s}.
\end{align}
Letting $p_{n+1}:=p_n+m+1$, and 
\[
\lambda_{n+1,k}:=\begin{cases}
	\lambda_{n,k}, & 1\leq k\leq p_n,\\
	\lambda+k-p_n-1& p_n<k\leq p_n+m+1,
\end{cases}
\]
we conclude from \eqref{casen+1} that \eqref{inequalitydenseset-s<1} is also valid with $n$ replaced by $n+1$.

The relation \eqref{inequalitydenseset2} is an immediate consequence of  \eqref{inequalitydenseset-s<1} and \eqref{differencevaluesf(x)}. 
\end{proof}

\subsection{Case $s=1$}

Let $\tau$ be the interval whose endpoints are the limit superior and the limit inferior of the sequence $(\Lambda(\eta(N)))_{N=1}^\infty$, and  let $F_1'$ and $\Lambda'$ respectively denote the sets of limit points of the sequences $\left(F_{N,1}\right)$ and $(\Lambda(\eta(N)))_{N=1}^\infty$. Then, $\tau=[-2\log 2,0]$, and what we want to show is that $F_1'=T_1$. Similar to how we argued at the beginning of Subsection \ref{sec:density1}, we can derive from \eqref{secondcase} the equivalence   $F_1'=T_1\iff \Lambda'=\tau$. 

By \eqref{doublingperiodicity2}, the set 
\[
\Lambda:=\{\Lambda(\eta(N)): N\in\mathbb{N}\}
\]
is contained in $\Lambda'$, and we have $\Lambda'\subset\tau$. So in order to prove that $\Lambda'=\tau$, it is  sufficient  to show that 	
\begin{align}\label{theo:densitys=1}\overline{\Lambda}=\tau.
\end{align} 

As in Subsection \ref{sec:density1}, for the justification of \eqref{theo:densitys=1} it is convenient to obtain first a different representation of the values $\Lambda(\eta(N))$. Suppose that $N\in\mathbb{N}$ has the binary decomposition \eqref{bindecompN}. Then
\begin{align*}
\Lambda(\eta(N)) & =\sum_{k=1}^{p}\frac{2^{n_{k}}}{N}\log\left(\frac{2^{n_{k}}}{N}\right)=\sum_{k=1}^{p}\frac{2^{n_{k}}}{N}\log\left(\frac{2^{n_{k}-n_{1}}}{N 2^{-n_1}}\right)\\
& =-\log(N 2^{-n_1})+\frac{1}{N 2^{-n_1}}\sum_{k=1}^{p} 2^{n_{k}-n_{1}}\log(2^{n_{k}-n_1})\\
& =-\log\left(\sum_{k=1}^{p}2^{n_{k}-n_1}\right)+\frac{1}{\sum_{k=1}^{p}2^{n_{k}-n_1}}\sum_{k=1}^{p} 2^{n_{k}-n_{1}}\log(2^{n_{k}-n_1}),
\end{align*}
so if we define $\lambda_{k}=n_{1}-n_{k}$, we have $0=\lambda_{1}<\lambda_{2}<\cdots<\lambda_{p}$ and
\begin{equation}\label{altrepLambdaetaN}
\Lambda(\eta(N))=-\log\left(\sum_{k=1}^{p}2^{-\lambda_{k}}\right)+\frac{1}{\sum_{k=1}^{p}2^{-\lambda_{k}}}\sum_{k=1}^{p} 2^{-\lambda_{k}}\log(2^{-\lambda_{k}}).
\end{equation}
It is therefore clear that
\[
\Lambda=\left\{-\log\left(\sum_{k=1}^{p}2^{-\lambda_{k}}\right)+\frac{\sum_{k=1}^{p} 2^{-\lambda_{k}}\log(2^{-\lambda_{k}})}{\sum_{k=1}^{p}2^{-\lambda_{k}}}: 0=\lambda_{1}<\cdots<\lambda_{p},\,\,\lambda_{i}\in\mathbb{Z}\right\}.
 \]

To simplify the forthcoming presentation, we will adopt the following notation. Given integers $0=\lambda_{1}<\lambda_{2}<\cdots<\lambda_{p}$, we define
\[
L_{p}(\lambda_{1},\ldots,\lambda_{p}):=-\log\left(\sum_{k=1}^{p}2^{-\lambda_{k}}\right)+\frac{\sum_{k=1}^{p} 2^{-\lambda_{k}}\log(2^{-\lambda_{k}})}{\sum_{k=1}^{p}2^{-\lambda_{k}}}.
\]

The equality \eqref{theo:densitys=1} is then an immediate consequence of the following result.

\begin{proposition}\label{prop:Lt}
Let $t$ be an interior point of the interval $\tau=[-2\log 2, 0]$ such that $t\notin \Lambda$. For each $n\in\mathbb{N}$, there exist integers $p_{n}\geq 1$, and $0=\lambda_{n,1}<\cdots<\lambda_{n,p_{n}}$, with $p_{n}<p_{n+1}$ for all $n\geq 1$, such that
\begin{equation}\label{ineqt}
L_{p_{n}+1}(\lambda_{n,1},\lambda_{n,2},\ldots,\lambda_{n,p_{n}},\lambda_{n,p_{n}}+1)<t<L_{p_{n}}(\lambda_{n,1},\lambda_{n,2},\ldots,\lambda_{n,p_{n}}).
\end{equation}
As a consequence, we have
\begin{equation}\label{limLtot}
\lim_{n\rightarrow\infty}L_{p_{n}}(\lambda_{n,1},\lambda_{n,2},\ldots,\lambda_{n,p_{n}})=t.
\end{equation}
\end{proposition}

In order to establish this proposition, we first prove a lemma.

\begin{lemma}\label{lemauxds1}
For any finite collection of $p\geq 1$ integers $0\leq \lambda_{1}<\cdots<\lambda_{p}$, the function
\[
f(x):=-\log\Big(x+\sum_{k=1}^{p}2^{-\lambda_{k}}\Big)+\frac{x\log x+\sum_{k=1}^{p} 2^{-\lambda_{k}}\log(2^{-\lambda_{k}})}{x+\sum_{k=1}^{p}2^{-\lambda_{k}}}
\]
is strictly decreasing on the interval $[0,2^{-\lambda_{p}}]$. We also have the estimate
\begin{equation}\label{estdiffineq}
f(0)-f(2^{-\lambda_{p}})\leq \log(1+2^{-\lambda_{p}})+2^{-\lambda_{p}}\log(2^{\lambda_{p}}).
\end{equation}
\end{lemma}
\begin{proof}
Set $A:=\sum_{k=1}^{p} 2^{-\lambda_{k}}\log(2^{-\lambda_{k}})$ and $B:=\sum_{k=1}^{p}2^{-\lambda_{k}}$. It is clear that $f$ is continuous on $[0,\infty)$, and for $x>0$ we have
\begin{gather*}
f'(x)=-\frac{1}{x+B}+\frac{(1+\log x)(x+B)-x\log x-A}{(x+B)^{2}}\\
=\frac{-x-B+(1+\log x)(x+B)-x\log x-A}{(x+B)^{2}}=\frac{B\log x-A}{(x+B)^{2}}\\
=\frac{\sum_{k=1}^{p}2^{-\lambda_{k}}\log(2^{\lambda_{k}}x)}{(x+B)^2}
\end{gather*}
therefore $f'(x)<0$ for all $x\in(0,2^{-\lambda_{p}})$, and $f$ is strictly decreasing on the interval $[0, 2^{-\lambda_{p}}]$.

We have
\[
f(0)-f(2^{-\lambda_{p}})=-\log(B)+\frac{A}{B}+\log(2^{-\lambda_{p}}+B)-\frac{2^{-\lambda_{p}}\log(2^{-\lambda_{p}})+A}{2^{-\lambda_{p}}+B}.
\]
Note that $A+2^{-\lambda_{p}}\log(2^{-\lambda_{p}})\leq 0$, and so
\[
\frac{2^{-\lambda_{p}}\log(2^{-\lambda_{p}})+A}{2^{-\lambda_{p}}+B}\geq \frac{2^{-\lambda_{p}}\log(2^{-\lambda_{p}})+A}{B},
\]
which implies that
\[
f(0)-f(2^{-\lambda_{p}})\leq \log\Big(1+\frac{2^{-\lambda_{p}}}{B}\Big)-\frac{2^{-\lambda_{p}}\log(2^{-\lambda_{p}})}{B}\leq \log(1+2^{-\lambda_{p}})+2^{-\lambda_{p}}\log (2^{\lambda_{p}})
\]
as desired.
\end{proof}

\begin{proof}[\textbf{\emph{Proof of Proposition \ref{prop:Lt}:}}] The proof is by induction on $n$. In virtue of \eqref{vector-for-eta(2^p-1)} and \eqref{altrepLambdaetaN}, for every integer $K\geq 1$ we have
\[
\Lambda(\eta(2^{K}-1))=-\log\left(\sum_{k=0}^{K-1}2^{-k}\right)+\frac{1}{\sum_{k=0}^{K-1}2^{-k}}\sum_{k=0}^{K-1} 2^{-k}\log(2^{-k})=L_{K}(0,1,\ldots,K-1),
\]
and so by Lemma \ref{lemauxds1} the sequence $(\Lambda(\eta(2^{K}-1)))_{K=1}^{\infty}$ is strictly decreasing. Observe that the initial value of this sequence is $0$ and it has limit $-2\log 2$ (see the proof of Theorem \ref{theo:main1}). Since $-2\log 2<t<0$ and $t\notin \Lambda$, it follows that there exists an integer $K_{1}\geq 1$ such that
\[
L_{K_{1}+1}(0,\ldots,K_{1}-1,K_{1})<t<L_{K_{1}}(0,\ldots,K_{1}-1).
\]
This shows that \eqref{ineqt} is valid for $n=1$ by taking $p_{1}=K_{1}$ and $\lambda_{1,k}=k-1$ for $1\leq k\leq p_{1}$.

Now we assume that \eqref{ineqt} holds for some integer $n\geq 1$. By Lemma \ref{lemauxds1}, the sequence 
\[
L_{p_{n}+1}(\lambda_{n,1},\ldots,\lambda_{n,p_{n}},\lambda_{n,p_{n}}+j),\qquad j=1, 2,\ldots
\]
is strictly increasing and its limit is $L_{p_{n}}(\lambda_{n,1},\ldots,\lambda_{n,p_{n}})$. Therefore there exists an integer $\lambda\geq \lambda_{n,p_{n}}+1$ such that
\begin{equation}\label{ineqst}
L_{p_{n}+1}(\lambda_{n,1},\ldots,\lambda_{n,p_{n}},\lambda)<t<L_{p_{n}+1}(\lambda_{n,1},\ldots,\lambda_{n,p_{n}},\lambda+1).
\end{equation}
Now observe that
\begin{gather}
L_{p_{n}+1}(\lambda_{n,1},\ldots,\lambda_{n,p_{n}},\lambda)\notag\\
=-\log\left(\sum_{k=1}^{p_{n}}2^{-\lambda_{n,k}}+2^{-\lambda}\right)+\frac{\sum_{k=1}^{p_{n}}2^{-\lambda_{n,k}}\log(2^{-\lambda_{n,k}})+2^{-\lambda}\log(2^{-\lambda})}{\sum_{k=1}^{p_{n}}2^{-\lambda_{n,k}}+2^{-\lambda}}\notag\\
=-\log\left(\sum_{k=1}^{p_{n}}2^{-\lambda_{n,k}}+\sum_{j=1}^{\infty}2^{-(\lambda+j)}\right)+\frac{\sum_{k=1}^{p_{n}}2^{-\lambda_{n,k}}\log(2^{-\lambda_{n,k}})+2^{-\lambda}\log(2^{-\lambda})}{\sum_{k=1}^{p_{n}}2^{-\lambda_{n,k}}+\sum_{j=1}^{\infty}2^{-(\lambda+j)}}\notag\\
>-\log\left(\sum_{k=1}^{p_{n}}2^{-\lambda_{n,k}}+\sum_{j=1}^{\infty}2^{-(\lambda+j)}\right)+\frac{\sum_{k=1}^{p_{n}}2^{-\lambda_{n,k}}\log(2^{-\lambda_{n,k}})}{\sum_{k=1}^{p_{n}}2^{-\lambda_{n,k}}+\sum_{j=1}^{\infty}2^{-(\lambda+j)}}\notag\\
+\frac{\sum_{j=1}^{\infty}2^{-(\lambda+j)}\log(2^{-(\lambda+j)})}{\sum_{k=1}^{p_{n}}2^{-\lambda_{n,k}}+\sum_{j=1}^{\infty}2^{-(\lambda+j)}}\label{eqineq}
\end{gather}
where the inequality follows easily from $\sum_{j=1}^{\infty}\frac{\lambda+j}{\lambda}\,2^{-j}>1$. 

By Lemma \ref{lemauxds1}, from \eqref{ineqst} and \eqref{eqineq} it follows that there exists an integer $m\geq 2$ such that
\[
L_{p_{n}+m}(\lambda_{n,1},\ldots,\lambda_{n,p_{n}},\lambda+1,\ldots,\lambda+m)<t<L_{p_{n}+m-1}(\lambda_{n,1},\ldots,\lambda_{n,p_{n}},\lambda+1,\ldots,\lambda+m-1)
\]
so if we set $p_{n+1}:=p_{n}+m-1$ and
\[
\lambda_{n+1,k}:=\begin{cases}
\lambda_{n,k}, & 1\leq k\leq p_{n},\\
\lambda+k-p_{n}, & p_{n}<k\leq p_{n}+m-1,
\end{cases}
\]
we have shown that \eqref{ineqt} is valid with $n$ replaced by $n+1$.

If we define the function
\[
f(x):=-\log\Big(x+\sum_{k=1}^{p_{n}} 2^{-\lambda_{n,k}}\Big)+\frac{x\log x+\sum_{k=1}^{p_{n}} 2^{-\lambda_{n,k}}\log(2^{-\lambda_{n,k}})}{x+\sum_{k=1}^{p_{n}}2^{-\lambda_{n,k}}}\,,
\]
then from \eqref{ineqt} and \eqref{estdiffineq} we deduce that
\begin{align*}
0\leq L_{p_{n}}(\lambda_{n,1},\ldots,\lambda_{n,p_{n}})-t & =f(0)-t\leq f(0)-f(2^{-(\lambda_{n,p_{n}}+1)})\leq f(0)-f(2^{-\lambda_{n,p_{n}}})\\
& \leq \log(1+2^{-\lambda_{n,p_{n}}})+2^{-\lambda_{n,p_{n}}}\log(2^{\lambda_{n,p_{n}}}),
\end{align*}
and since $\lambda_{n,p_{n}}\rightarrow\infty$ as $n\rightarrow\infty$, \eqref{limLtot} follows.
\end{proof}

\bigskip

\noindent \textsc{Department of Mathematics, University of Central Florida, 4393 Andromeda Loop North, Orlando, FL 32816, USA} \\
\textit{Email address}: \texttt{abey.lopez-garcia\symbol{'100}ucf.edu}

\bigskip

\noindent \textsc{Department of Mathematics, The University of Mississippi, Hume Hall 305, P.O. Box 1848, University, MS 38677, USA}\\
\textit{Email address}: \texttt{minadiaz\symbol{'100}olemiss.edu}

\end{document}